\documentclass[12pt]{amsart}
\usepackage{amsmath, amssymb, color}
\usepackage{amsthm}
\usepackage[dvipdfmx]{graphicx}
\usepackage{multirow,bigdelim}
\usepackage{array}
\usepackage{epsfig}
\usepackage{amscd}
\numberwithin{equation}{section}
\theoremstyle{definition}
\newtheorem{thm}{Theorem}[section]
\newtheorem{cor}[thm]{Corollary}

\newtheorem{defn}[thm]{Definition}

\newtheorem{rem}[thm]{Remark}

\newtheoremstyle{mydefinition}
  {}   
  {}   
  {}  
  {0pt}       
  {} 
  {.}         
  {5pt plus 1pt minus 1pt} 
  {}          

\newtheoremstyle{citing}
  {}
  {}
  {\itshape}
  {}
  {\bfseries}
  {.}
  {.5em}
  {\thmnote{#3}}

\theoremstyle{citing}
\newtheorem*{citingtheorem}{}

\makeatother

\newcommand{\Z}{\mathbb{Z}}
\newcommand{\N}{\mathbb{N}}
\newcommand{\Xmn}{X^{m,n}}
\newcommand{\Kmn}{K^{m,n}}

\newcommand{\Em}{E_m}
\newcommand{\Eam}{E^{\ast}_m}
\newcommand{\En}{E_n}
\newcommand{\Ean}{E^{\ast}_n}
\newcommand{\ve}{\varepsilon}

\title{Gluck twists along 2-knots with periodic monodromy}
\author{Mizuki Fukuda}
\address{Mathematical Institute, Tohoku University, Sendai, 980-8578, Japan}
\email{fukuda.mizuki.r5@dc.tohoku.ac.jp}

\begin{document}
\maketitle
\vspace{-5mm}
\begin{abstract}
The union of singular orbits of an effective locally smooth circle action on the 4-sphere consists of two 2-knots, $K$ and $K^{\prime}$, intersecting at two points transversely.
Each of $K$ and $K^{\prime}$ is called a branched twist spin. A twist spun knot is an example of a branched twist spin. 
The Gluck twists along branched twist spins are studied by Fintushel, Gordon and Pao. 
In this paper, we determine the 2-knot obtained from $K$ by the Gluck twist along $K^{\prime}$.
As an application, we give infinitely many pairs of inequivalent branched twist spins whose complements are homeomorphic.
\end{abstract}

\section{Introduction}
A branched twist spin is a 2-knot in $S^4$ that is invariant under an effective locally smooth $S^1$-action on $S^4$, which is studied by Pao in~\cite{Pa} and then Hillman and Plotnick in~\cite{HP}. 
It was not obvious that the total space of a branched twist spin is $S^4$.
Historically, Montgomery and Yang showed that effective locally smooth $S^1$-actions are classified into 
 four types~\cite{MY},  Fintushel showed that 
 there is a bijection between orbit data and weak equivalence classes of $S^1$-actions on homotopy $4$-spheres~\cite{Fi}, 
 and then Pao showed that the homotopy $4$-sphere is actually the $4$-sphere~\cite{Pa} by using a certain induction.

We introduce a branched twist spin briefly.
Suppose that $S^1$ acts locally smoothly and effectively on $S^4$ and the orbit space is $S^3$.
Then there are at most two types of exceptional orbits called $\Z_m$-type and $\Z_n$-type, where $m,n$ are coprime positive integers.  
Let $E_m$ (resp. $E_n$) be the set of exceptional orbits of $\Z_m$-type (resp. $\Z_n$-type) and $F$ be the fixed point set.
The image of the orbit map of $E_n$, denoted by $E^{\ast}_n$, is an open arc in the orbit space $S^3$,
 and that of $F$, denoted by $F^{\ast}$, is the two points in $S^3$ which are the end points of $E^{\ast}_n$.
It is known that $E^{\ast}_m \cup E^{\ast}_n \cup F^{\ast}$ constitutes a $1$-knot $K$ in $S^3$ 
 and $E_n \cup F$ is diffeomorphic to the 2-sphere.  
The {\it $(m,n)$-branched twist spin} of $K$ is defined as $E_n \cup F$.
Note that it is known by Plotnick that a fibered $2$-knot is a branched twist spin if and only if its monodromy is periodic~\cite{Pl}. Therefore, this class has special importance among other known classes of fibered $2$-knots. 
Also note that spun knots and twist spun knots are included in the class of branched twist spins.
The definition of a branched twist spin is generalized for $(m,n)\in \mathbb{Z} \times \mathbb{N}$ in~\cite{F} by taking the orientation of $S^4$ and that of the $S^1$-action into account, see Section 2.1. We denote it by $K^{m,n}$.

The Gluck twist is one of the important surgeries of 4-manifolds.
It is known that a manifold obtained from $S^4$ by the Gluck twist along a 2-knot is a homotopy 4-sphere.
However, it is still a question whether the Gluck twist along a 2-knot yields again $S^4$ or not~\cite{Ki}.
In 1976, Gordon showed that the manifold obtained from $S^4$ by the Gluck twist along an $m$-twist spun knot $K^{m,1}$ is always $S^4$ for any $K$ and integer $m$~\cite{G},
and Pao showed the same statement for all branched twist spins implicitly~\cite{Pa}.

The aim of our study is to clarify which 2-knot can be obtained by the Gluck twist along a branched twist spin.
Our main theorem is the following:
\begin{citingtheorem}[Theorem \ref{thm2}]
Let $(m,n) \in (\Z \setminus \{0\}) \times \N$ be a coprime pair. Then $K^{m+n, n}$ is obtained from $K^{m,n}$ by the Gluck twist along $K^{\ve n, \ve m}$, where $\ve= 1$ if $m>0$ and $\ve= -1$ if $m<0$.
\end{citingtheorem}
In consequence, we see that the 2-knot treated in the paper of Gordon \cite{G} is $K^{\ve^{\prime}m,\ve^{\prime}(m+1)}$,
where $\ve^{\prime} =1$ if  $m+1>0$ and $\ve^{\prime}= -1$ if $m+1<0$ .
We remark that Pao used a surgery along a branched twist spin to show that the total space is $S^4$.
Actually this surgery is nothing but the Gluck twist observed in Theorem~\ref{thm2}.
The proof of 
Theorem~\ref{thm2} is based on the Gordon's method in~\cite{G}.
The 4-sphere decomposes into five connected pieces with respect to the orbit data of the $S^1$-action. 
Both of the complements of $\Kmn$ and $K^{\ve n,\ve m}$ can be obtained by gluing some of these pieces, and filling the remaining pieces correspond to the Gluck surgeries.
The proof is done by observing this decomposition more precisely with information of the $S^1$-action.
As a corollary, we have the following. 
\begin{citingtheorem}[Corollary \ref{cor}]
Assume that $m$ is odd and $K^{m,n}$ is non-trivial. Then $K^{m,n}$ and $K^{\ve^{\prime}m,\ve^{\prime}(m+n)}$ are not equivalent but their complements are homeomorphic, 
where $\ve^{\prime}= 1$ if $m+n>0$ and $\ve^{\prime}= -1$ if $m+n<0$.
\end{citingtheorem}

This paper is organized as follows:
In section 2, we give a decomposition of $S^4$ with respect to the orbit data of an $S^1$-action, define the $(m,n)$-branched twist spin as an oriented $2$-knot, introduce the Gluck twist briefly and set notations that will be used in Section 3.  
In Section 3, we explicitly explain the Gluck twist along an $(m,n)$-branched twist spin by showing the gluing maps of the pieces of the decomposition and prove the main theorem.
\ \\
\ \\
{\bf Acknowledgements.}
The author would like to express his deep gratitude to Masaharu Ishikawa for the helpful suggestions and his encouragement. 
 The author 
 is 
 supported by JSPS KAKENHI Grant Number 18J11484.

\section{Preliminaries}
\subsection{Branched twist spin}
Suppose that $S^4$ has an effective locally smooth $S^1$-action.
Let $E_m$ denote the set of exceptional orbits of $\mathbb{Z}_m$-type, where $m$ is a positive integer, 
and $F$ be the fixed point set.
Let $E^{\ast}_m$ and $F^{\ast}$ denote the image of $E_m$ and $F$ by the orbit map, respectively.
Montgomery and Yang showed that effective locally smooth $S^1$-actions are classified into 
the following four types:
(1) $\{D^3\}$,(2) $\{S^3\}$, (3) $\{S^3 , m\}$, (4) $\{ (S^3, K), m ,n\}$, which are called orbit data~\cite{MY}. 
The $3$-ball and the $3$-sphere in these notations represent the orbit spaces.
In case (4), the union $E^{\ast}_m \cup E^{\ast}_n \cup F^{\ast}$ constitutes a $1$-knot $K$ in the orbit space $S^3$ 
and the union $E_n \cup F$ constitutes a 2-knot in $S^4$.
In case (3), for an arc $A^{\ast}$ in $S^3$ whose end points are $F^{\ast}$, the union of the preimage of $A^{\ast}$ and $F$ constitutes a 2-knot in $S^4$, which is called a twist spun knot.

We recall the definition of $(m,n)$-branched twist spins for $(m,n)\in\mathbb{Z}\times\mathbb{N}$ in~\cite{F}.
Fix the orientations of $S^4$ and the $S^1$-action.
Set $K = \Eam \cup \Ean \cup F^{\ast}$, which is a 1-knot in $S^3$, and $N(K)$ to be a compact tubular neighborhood of $K$.
Let $(m,n)$ be a pair of integers in $(\mathbb{Z} \setminus \{0\}) \times\mathbb N$
such that $|m|$ and $n$ are coprime. 
We decompose the orbit space $S^3$ into five connected pieces as follows:
\begin{equation}\label{matrix}
S^3 =  (D^{3\ast}_1\sqcup D^{3\ast}_2) \cup (E^{c\ast}_n \times D^2)  \cup (E^{c\ast}_m \times D^2) \cup X,
\end{equation}
where $X = S^3 \setminus {\rm int}N(K)$, $D^{3\ast}_1 \sqcup D^{3\ast}_2$ is a compact neighborhood of $F^{\ast}$ and $E^{c\ast}_m$ and $E^{c\ast}_n$ are the connected components of $K \setminus \text{int}(D^{3\ast}_1 \cup D^{3\ast}_2)$ 
such that  $E_m^{c\ast}\subset E_m^{\ast}$ and $E_n^{c\ast}\subset E_n^{\ast}$, see Figure~\ref{compK}.
Using this decomposition, we have a decomposition of $S^4$ as follows: 
Let $p: S^4 \to S^3$ denote the orbit map.
Each point of $X$ is the image of a free orbit.
Thus the preimage $p^{-1}(X)$ is diffeomorphic to $X \times S^1$ and $p|_{X\times S^1} : X\times S^1 \to X$ is the first projection 
since $H^2(X;\Z) = 0$. 
Let $B^4_i$ be a linear slice at the fixed point in $p^{-1}(D^{3\ast}_i)$, which is a closed 4-ball, for $i=1,2$.
Choosing a point $z^{\ast}_m \in E^{c\ast}_m$, let $D^{2\ast}_{z^{\ast}_m}$ be a 2-disk in $S^3$ centered at $z^{\ast}_m$ and transversal to $E^{c\ast}_m$.
The preimage $p^{-1}(D^{2\ast}_{z^{\ast}_m})$ is a solid torus $V_m$ whose core is an exceptional orbit of $\mathbb{Z}_m$-type.
In the same way, choosing a point $z^{\ast}_n \in E^{c\ast}_n$, 
let $D^{2\ast}_{z^{\ast}_n}$ be a 2-disk in $S^3$ centered at $z^{\ast}_n$ and transversal to $E^{c\ast}_n$.
The preimage $p^{-1}(D^{2\ast}_{z^{\ast}_n})$ is a solid torus $V_n$ whose core is an exceptional orbit of $\mathbb{Z}_n$-type.
Then $S^4$ is decomposed into the five connected pieces:
\begin{equation}\label{dec}
S^4 = (B^4_1 \sqcup B^4_2)  \cup(V_n \times E^{c\ast}_n)  \cup  (V_m \times E^{c\ast}_m)\cup (X \times S^1).
\end{equation}
\begin{figure}[htbp]
\centering
\includegraphics[scale=0.3]{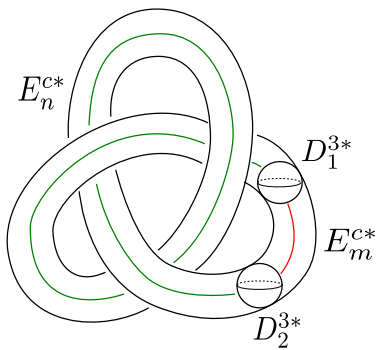}
\caption{Decomposition of $N(K)$}
\label{compK}
\end{figure}
\begin{defn}[Branched twist spin]
Let $K$ be an oriented knot in $S^3$.
For each pair $(m,n)\in\mathbb{Z}\times \mathbb{N}$ with 
$m\ne 0$ such that $|m|$ and $n$ are coprime,
let $K^{m,n}$ denote the $2$-knot $E_n\cup F$.
If $(m,n)=(0,1)$ then define $K^{0,1}$ to be the spun knot of $K$.
The $2$-knot $K^{m,n}$ is called the $(m,n)$-$branched\ twist\ spin$ of $K$.
\end{defn}

Note that the branched twist spin $K^{m,1}$ constructed from $\{(S^3 , K),m,1\} $ is the $m$-twist spun knot of $K$.

\subsection{Gluck twist}
Let $M$ be an $(n+2)$-dimensional manifold and $L$ be an $n$-knot in $M$.
We can construct a new manifold by removing an open neighborhood ${\rm int} N(L)$ of $L$ and regluing $S^n \times D^2$ with some homeomorphism 
$\gamma: S^n \times S^1 \to S^n \times S^1$.
The pair  $(M, S^n)$ only depends on the homeotopy class of  $\gamma$, and Gluck showed the group of homeotopy classes is isomorphic to $\Z_2 \times \Z_2 \times \Z_2$ if $n \geq 2$~\cite{Gl}.
The first two factors correspond to the orientation-reversals of $S^n$ and $S^1$ and the last factor is generated by $\sigma =\nu \cup \nu^{\prime}$ defined by
$$
\nu (x,(r,\theta), \phi) = (x, (r, \theta - \phi), \phi)\ \ \ \ \ (x,(r,\theta),\phi) \in (\partial D^{n-1} \times D^2) \times S^1 \subset S^{n} \times S^1,
$$
$$
\nu^{\prime} (x,\theta , \phi) = (x, \theta - \phi, \phi)\ \ \ \ \  (x, \theta, \phi) \in (D^{n-1} \times \partial D^2) \times S^1\subset S^{n} \times S^1.
$$
The map $\sigma$ represents the twist of $S^n$ about its polar $S^{n-2}$ once while $S^n$ rotates one time along $S^1$.
The operation that constructs a new manifold by removing an open neighborhood ${\rm int} N(L)$ of $L$ and regluing $S^n \times D^2$ with $\sigma$ is called the Gluck twist of $M$ along $L$.
We denote by $\Sigma(L)$ the manifold obtained by the Gluck twist of $S^4$ along $L$.

For twist spun knots,  Gordon showed $\Sigma(K^{m,1}) = \Sigma(K^{m+1,1})$  by constructing these manifolds from the pieces in~\eqref{dec}.

\section{Proof of Theorem 3.1.}
In this section, we assume that $m \neq 0$, that is, we only consider the $S^1$-action whose orbit data is $\{(S^3,K),m,n\}$.
Note that the orbit data $\{S^3,m\}$ is regarded as $\{(S^3,K),m,1\}$ by taking an arc $A^{\ast}$ in the images of free orbits in $S^3$ so that the union of $E^{\ast}_m$ and $A^{\ast}$ is $K$.

\subsection{Decomposition of $S^4$ along $\Kmn$}

Let $N(K^{m,n})$ be a compact tubular neighborhood of $K^{m,n}$ and  $X^{m,n} = S^4 \setminus \text{int}N(K^{m,n})$ be the knot complement of $K^{m,n}$.
In this section, we construct $X^{m,n}$ from the pieces in~\eqref{dec} by defining attaching maps concretely. 

From the given orientations of $S^4$ and the $S^1$-action,
we first fix coordinates on $\partial X \times S^1$ as follows.
Let $(\theta, x)$ be coordinates on $\partial X \cong \partial D^2 \times S^1$ such that $\theta$ is the meridian and $x$ is the longitude of $K$ in $S^3$.
The coordinates on $\partial X \times [0,1] \subset X$ are given by $(\theta, x, y)$, where $\partial X$ lies in $\partial X \times \{0\}$.
Reversing the direction of $\theta$ if necessary, we may assume that the coordinates $(\theta, x, y, h)$ on $\partial X \times [0,1] \times S^1$ are consistent with the orientation of $S^4$, where $h$ is the direction of the $S^1$-action.
The orientation of $X$ is given by $(\theta,x,y)$.
Then, using the projection $\partial X \times [0,1] \times S^1\to \partial X \times S^1$, we have the coordinates $(\theta, x, h)$ on $\partial X \times S^1$.
Next we define coordinates on $\partial V_m \times E^{c\ast}_m$. 
The 4-ball $B^4_1$ is homeomorphic to $D^2_m \times D^2_n$, where $D_m^2$ (resp. $D_n^2$) is a 2-disk whose center corresponds to the exceptional orbit of $\Z_m$-type (resp. $\Z_n$-type).
The boundary $\partial B^4_1$ is homeomorphic to $V_ m \cup V_n$, where $V_m = D_m^2 \times \partial D_n^2$ and $V_n = \partial D_m^2 \times D_n^2$.
Let $(r_1,\theta_1)$ be polar coordinates of $D^2_m$ and $(r_2 ,\theta_2)$ be polar coordinates of $D^2_n$ such that $(r_1, \theta_1, r_2, \theta_2)$ are consistent with the orientation of $S^4$.
We may choose the indices of $B_1$ and $B_2$ such that the direction of $x$ is from the origin of $B_1$ to that of $B_2$ through $E^{c\ast}_m$, see Figure~\ref{twins}.
The coordinates on $\partial V_m \times E^{c\ast}_m$ are given as $(\theta_1, r_2, \theta_2)$ and $r_2$ coincides with $x$ on $\partial V_m \times E^{c\ast}_m$.
Thus the coordinates on $\partial V_m \times E^{c\ast}_m$ are given as $(\theta_1, x, \theta_2)$.

As we mentioned in Section 2.1, the free orbits are curves on $\partial V_m \cong \partial D_m^2 \times \partial D_n^2$ rotating, up to orientation, $m$ times along $\partial D_n^2$ and $n$ times along $\partial D_m^2$. 
Changing the coordinates $(\theta_1, \theta_2)$ into $(-\theta_1, -\theta_2)$ if necessary, 
we assume that the free orbits on $\partial V_m \cong \partial D_m^2 \times \partial D_n^2$ are rotating $\ve m$ times along $\partial D_n^2$ and $\ve n$ times along $\partial D_m^2$,  where $\ve = 1$ if $m\geq0$ and $\ve = -1$ if $m < 0$.  
 Comparing the free orbits in $\partial D^2_m \times E^{c\ast}_m \times \partial D^2_n$ and $\partial X \times S^1$, 
 we can see that the gluing map $g : \partial D^2_m \times E^{c\ast}_m \times \partial D^2_n \to \partial X \times S^1$
 of $ D^2_m \times E^{c\ast}_m \times \partial D^2_n$ and $ X \times S^1$, that yields $X^{m,n}$, satisfies the equality
 $$
 (\theta, x , h) = g(\alpha \theta + \ve n h, x , -\beta \theta + \ve m h),
 $$
 where $\alpha$ and $\beta$ are integers satisfying $m\alpha + n \beta = \ve$.
Since the orientations of $(\theta, x, y, h)$ and $(r_1, \theta_1, r_2, \theta_2)$ coincide with that of $S^4$, 
the map $g$ is orientation preserving.
In terms of the coordinates $(\theta_1, x, \theta_2)$, we can write $g$ as
$$
g(\theta_1, x , \theta_2) = (\ve m \theta_1 - \ve n \theta_2, x, \beta \theta_1 +\alpha \theta_2).
$$
Let $c_{\theta}, c_{h} ,c_{\theta_1}$ and $c_{\theta_2}$ be 1-cycles given by
$$
c_{\theta} = \{(\theta,\hat{x},\hat{h})\in \partial X \times S^1\ |\ \theta \in [0,2\pi]\},
$$
 $$
c_{h}= \{(\hat{\theta}, \hat{x}, h)\in \partial X \times S^1 \ |\ h \in [0,2\pi]\},
$$
$$
c_{\theta_1}= \{(\theta_1, \hat{x} ,\hat{\theta}_2)\in \partial V_m \times E^{c\ast}_m\ |\ \theta_1 \in [0,2\pi]\},
$$
$$
c_{\theta_2}= \{(\hat{\theta}_1, \hat{x} ,\theta_2)\in \partial V_m \times E^{c\ast}_m\ |\ \theta_2 \in [0,2\pi]\},
$$
where $\hat{x}, \hat{h}, \hat{\theta}_1$ and $\hat{\theta}_2$ are constants.
These cycles are oriented according to the coordinates $\theta, h,\theta_1$ and $\theta_2$.
By using these 1-cycles, the induced map 
$g_{\ast}: H_1(\partial D^2_m \times E_m^{c\ast}\times \partial D^2_n)\to H_1((\partial X \times S^1) \cap (\partial V_m\times E_m^{c\ast}))$ 
 is represented as
\begin{equation}\label{matrix}
(g_{\ast}(c_{\theta_1}), g_{\ast}(c_{\theta_2}) ) = 
(c_{\theta}, c_h)
\begin{pmatrix}
\ve m & -\ve n\\
\beta & \alpha
\end{pmatrix},
\end{equation}
which corresponds to $f^{-1}_{\ast}$ defined in~\cite{F}.

\begin{figure}[htbp]\label{figure3}
\centering
\includegraphics[scale= 0.5]{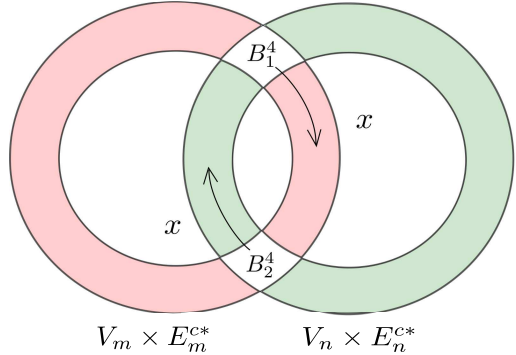}
\caption{The coordinate $x$ on $V_m \times E^{c\ast}_m$  and $V_n \times E^{c\ast}_n$}
\label{twins}
\end{figure}

In summary, we have
\begin{equation}\label{compKmn}
\begin{aligned}
S^4 &= N(\Kmn) \cup \Xmn \\
&= ((B^4_1 \cup B^4_2)\cup(\partial D^2_m \times E^{c\ast}_n \times D^2_n) )  \cup \Xmn \\
&=   ((B^4_1 \cup B^4_2)\cup(\partial D^2_m \times E^{c\ast}_n \times D^2_n) )\cup
((D^2_m \times E^{c\ast}_m \times \partial D^2_n) \cup_g (X \times S^1)).
\end{aligned}
\end{equation}

Next we define the gluing map of $\partial D^2_m \times E^{c\ast}_n \times D^2_n $ and $X \times S^1$.
As we did for $B^4_1$,  the coordinates  of  $B^4_2\cong {D^{2}_m}^{\hspace{-2pt} \prime} \times {D^2_n}^{\prime}$  is defined as 
($r^{\prime}_1,\theta^{\prime}_1,r^{\prime}_2,\theta^{\prime}_2)$, where $(r^{\prime}_1,\theta^{\prime}_1)$ and $(r^{\prime}_2,\theta^{\prime}_2)$ 
are polar coordinates on  ${D^{2}_m}^{\hspace{-2pt} \prime}$ and  ${D^2_n}^{\prime}$, respectively.
Since ${D^{2}_m}^{\hspace{-2pt} \prime} \times \partial {D^2_n}^{\prime} \subset  D^2_m \times E^{c\ast}_n \times \partial D^2_n$
 we may assume that $r_1 = r^{\prime}_1, \theta_1 = \theta^{\prime}_1$  and $\theta_2 = \theta^{\prime}_2$.
Comparing the  coordinates $r_2$  of $(r_1,\theta_1,r_2,\theta_2)$ and  $r^{\prime}_2$ of $(r^{\prime}_1,\theta^{\prime}_1,r^{\prime}_2,\theta^{\prime}_2)$,  we have $r_2 = -r^{\prime}_2 $, see Figure~\ref{direct}.
Then, as we did for $g$, the gluing map $e: \partial D^2_m \times E^{c\ast}_n \times \partial D^2_n \to \partial X \times S^1 $ is defined as 
$$
e(\theta^{\prime}_1,r^{\prime}_2,\theta^{\prime}_2) = (\ve m \theta^{\prime}_1 - \ve n \theta^{\prime}_2, r^{\prime}_2, \beta \theta^{\prime}_1 +\alpha \theta^{\prime}_2).
$$
Replacing $x$ by $r_2'$, we have 
\begin{equation}
e(\theta_1,x,\theta_2) = (\ve m \theta_1 - \ve n \theta_2, x, \beta \theta_1 +\alpha \theta_2).
\end{equation} 
Since the disjoint union of 4-balls $B^4_1 \cup B^4_2$ is a cone of the union of $V_n \times \partial E^{c\ast}_n$ and $V_n \times \partial E^{c\ast}_n$,
 $B^4_1 \cup B^4_2$ and $(\partial D^2_m \times E^{c\ast}_n \times D^2_n) \cup_e X^{m,n}$ are glued identically along $\partial B^4_1 \cup \partial B^4_2$.
 This means $e \cup {\rm id} $ is the attaching map of $(\partial D^2_m \times E^{c\ast}_n \times D^2_n) \cup  (B^4_1 \cup B^4_2)$ and  $ X^{m,n}$, that yields $S^4$.

\begin{figure}[htbp]
\centering
\includegraphics[scale=0.45]{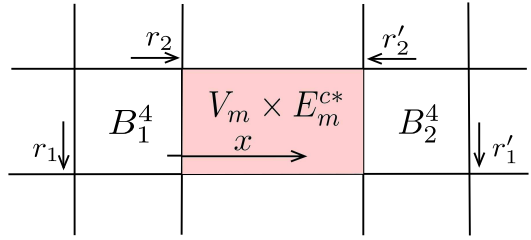}
\caption{The coordinates $r_1, r_2, r^{\prime}_1,r^{\prime}_2$ and $x$}
\label{direct}
\end{figure}

\subsection{Gluck twist along $K^{m,n}$}

We give a decomposition of $\Sigma(K^{m,n})$ from the decomposition \eqref{compKmn}.
By definition, 
$K^{m,n}$ is the core of the union of $D^2_m \times \partial F^{\ast} \times D^2_n$ and $\partial D^2_m \times E^{c\ast}_n \times D^2_n$ and
$\Sigma(K^{m,n})$ is diffeomorphic to $(S^2 \times B^2) \cup_{\sigma} X^{m,n}$.
Here
$$
S^2 \times B^2 \cong (D^2 \times \partial I \times B^2) \cup_{{\rm id}_{\partial D^2 \times \partial I \times  B^2}} (\partial D^2 \times I \times B^2)
$$
is regarded as 
 $$
 (D^2_m \times \partial E^{c\ast}_n \times D^2_n) \cup_{{\rm id}_{\partial D^2_m \times \partial E^{c\ast}_m \times  D^2_n}} (\partial D^2_m \times E^{c\ast}_n \times D^2_n),
 $$
 where ${\rm id}_{\partial D^2 \times \partial I \times  B^2}$ is the canonical gluing map on $\partial D^2 \times \partial I \times  B^2$ and ${\rm id}_{ \partial D^2_m \times \partial E^{c\ast}_n \times  D^2_n}$ is that on $\partial D^2_m \times \partial E^{c\ast}_n \times  D^2_n$.
We use either $D^2$ or $B^2$ instead of  $D^2_{m}, D^2_{n}$ and use $I$ instead of $E^*$
since we do not define the $S^1$-action on 
the pieces $(D^2 \times \partial I \times B^2) \cup_{{\rm id}_{\partial D^2 \times \partial I \times  B^2}} (\partial D^2 \times I \times B^2)$ yet.
Note that, before the Gluck twist,  
$D^2 \times \partial I \times \partial B^2$ and $D^2_m  \times \partial E^{c\ast}_m \times  \partial D^2_n$ are glued identically.
We denote this gluing map by $e^{\prime\prime}$.
The Gluck twist $\sigma$ is the union of  $\nu$ and $\nu^{\prime}$ that are defined by
$$
\nu((r_1,\theta_1),x, \theta_2) = ((r_1, \theta_1 - \theta_2), x, \theta_2)\ \  {\rm on}\ \  D^2 \times \partial I \times \partial B^2,
$$
$$
\nu^{\prime}(\theta_1,x,\theta_2) = (\theta_1 - \theta_2, x, \theta_2) \ \ {\rm on}\ \  \partial D^2 \times I \times \partial B^2.
$$
Using these maps, the decomposition of  $\Sigma(\Kmn)$ is given as
\begin{equation}\label{sigma}
\left((D^2 \times \partial I \times B^2) \cup_{{\rm id}_{\partial D^2 \times \partial I \times  B^2}} (\partial D^2 \times I \times B^2)\right)
 \cup_{\lambda \cup \lambda^{\prime}} 
 \left((D_m^2 \times E_m^{c\ast}\times \partial D_n^2) \cup_g (X \times S^1) \right),
\end{equation}
where $\lambda =  e^{\prime\prime} \circ  \nu$ and $\lambda^{\prime} = e \circ \nu^{\prime}$.
Focusing on the coordinates of $\partial D^2 \times \partial B^2$, we may represent the map $\nu \cup \nu^{\prime}$ by the matrix
$\begin{pmatrix}
1 & -1\\
0 & 1\\
\end{pmatrix}$ since it is given by
$
(c_{\theta_1}, c_{\theta_2}) \mapsto (c_{\theta_1}, c_{\theta_2}) 
\begin{pmatrix}
1 & -1\\
0 & 1\\
\end{pmatrix}$.
Therefore $\lambda$ and $\lambda^{\prime}$ are represented by the matrices
$
\begin{pmatrix}
1 & -1 \\
0 & 1 \\
\end{pmatrix}
$ and 
$
\begin{pmatrix}
\ve m & -\ve n \\
\beta & \alpha \\
\end{pmatrix}
\begin{pmatrix}
1 & -1\\
0 & 1
\end{pmatrix},$
 respectively.

The next theorem is proved by Pao in~\cite{Pa}.
We give a more precise proof of this assertion which will be used in the proof of Theorem~\ref{thm2}.
\begin{thm}[Pao]\label{thm1}
Let $(m,n) \in \Z \times \N$ be a coprime pair. Then $\Sigma({K^{m,n}})$ is diffeomorphic to $S^4$.
\end{thm}
\begin{proof}
Since it is known by Gordon that $\Sigma(K^{0,1})$ is diffeomorphic to $S^4$ \cite{G}, we assume that $m\neq 0$. 
The proof is based on the Gordon's argument in~\cite{G}, that is, 
we rearrange the decomposition~\eqref{sigma} to another decomposition
\begin{equation}\label{sigma2}
\left( (D^2 \times \partial I \times B^2) \cup_{\hat{\rm id}} (\partial D^2_m \times E^{c\ast}_m \times D^2_n)\right)
\cup_{\mu \cup \mu^{\prime}}
 \left((D^2 \times I \times \partial B^2) \cup_{\tilde{g}} (X \times S^1)\right),
\end{equation}
where $\mu$ and $\mu^{\prime}$ will be chosen such that the glued 4-manifold becomes $\Sigma(K^{m+n,n})$.
See Figure~\ref{arrange} for this rearrangement.

\begin{figure}[htbp]
\centering
\includegraphics[scale=0.68]{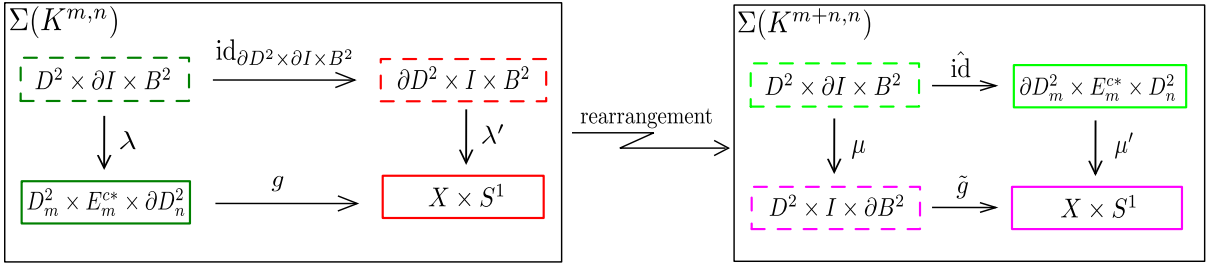}
\caption{Decompositions of $\Sigma(K^{m,n})$ and $\Sigma(K^{m+n,n})$}
\label{arrange}
\end{figure}

First, we focus on $(\partial D^2 \times I \times B^2) \cup_{\lambda^{\prime}} (X \times S^1)$. See the box of First step in Figure~\ref{maps}.
Let ${\rm id}_{X \times S^1}$ be the identity map on $X \times S^1$ and $u : D^2 \times I \times \partial B^2 \to \partial D^2 \times I \times B^2 $ be the map defined by
$$
u((r_1,\theta_1),x,\theta_2) = (\theta_2, x, (r_1, -\theta_1 + \theta_2)),
$$
which corresponds to the matrix
$\begin{pmatrix}
0 & 1 \\
-1 & 1 
\end{pmatrix}$.
Set ${\rm id}_{\partial X \times S^1}$ to be the restriction of ${\rm id}_{X\times S^1}$ to ${\partial X \times S^1}$ 
and set $\hat{u} = u|_{\partial D^2 \times I \times \partial B^2}$.
 Then $u^{-1}\cup {\rm id}_{X \times S^1} $ is a homeomorphism from
  $(\partial D^2 \times I \times B^2) \cup_{\lambda^{\prime}} (X \times S^1)$
  to 
 \begin{equation} \label{pieceA}
 (D^2 \times I \times \partial B^2) \cup_{\tilde{g}} (X \times S^1),
 \end{equation}
  where $\tilde{g} ={\rm id}_{X \times S^1} \circ \lambda^{\prime}\circ \hat{u}$.
We can check that $\tilde{g}$ corresponds to the matrix
$$
\begin{pmatrix}
1 & 0 \\
0 & 1 \\
\end{pmatrix}
\begin{pmatrix}
\ve m & -\ve n \\
\beta & \alpha \\
\end{pmatrix}
\begin{pmatrix}
1 & -1 \\
0 & 1 \\
\end{pmatrix}
\begin{pmatrix}
0 & 1 \\
-1 & 1 \\
\end{pmatrix}
=
\begin{pmatrix}
\ve(m + n) & -\ve n \\
-\alpha +\beta & \alpha \\
\end{pmatrix}.
$$

If $m<0$ and $|m|<n$ then the left-top entry $\ve(m+n)$ of the above matrix is negative.
In this case, replacing $(\theta_1,x, \theta_2)$ with $(-\theta_1,x, -\theta_2)$,
we may change the matrix 
$\begin{pmatrix}
\ve(m + n)& -\ve n \\
-\alpha + \beta & \alpha  \\
\end{pmatrix}
$
into
$
\begin{pmatrix}
-\ve(m+n)& \ve n \\
\alpha - \beta & -\alpha  \\
\end{pmatrix}=
\begin{pmatrix}
\ve^{\prime}(m+n)& -\ve^{\prime} n \\
\alpha - \beta & -\alpha  \\
\end{pmatrix}
$, where $\ve^{\prime} = 1$ if $m+n>0$ and $\ve^{\prime} = -1$ if $m+n<0$, so that the left-top entry of the above matrix becomes positive.
This is necessary since the matrix presentation of the decomposition in \eqref{matrix} is given with this property.

Next we focus on $(D^2 \times \partial I \times B^2) \cup_{\lambda} (D_m^2 \times E_m^{c\ast}\times \partial D_n^2)$. 
See the box of Second step in Figure~\ref{maps}.
Let
$v: D^2_m \times E^{c\ast}_m \times \partial D^2_n \to \partial D^2_m \times E^{c\ast}_m \times D^2_n$ 
and 
$w : D^2 \times \partial I \times B^2 \to D^2 \times \partial I \times B^2$
 be the homeomorphisms given by 
$\begin{pmatrix}
0 & -1 \\
1 & 1 \\
\end{pmatrix}$
and
$\begin{pmatrix}
0 & 1 \\
-1 & 0 \\
\end{pmatrix}$, respectively.
Set
$\hat{v} = v| _{D^2_m \times \partial E^{c\ast}_m \times \partial D^2_n}$ and 
$\hat{w} = w|_{\partial D^2 \times \partial I  \times B^2}$.
Note that $w$ is chosen so as to exchange the basis of $ D^2$ and $B^2$ in $ D^2 \times B^2$ and
$v$ is chosen so that $\hat{v}\circ \lambda \circ \hat{w} = \hat{{\rm id}}$, 
where $\hat{{\rm id}}$ is the identification from $\partial D^2 \times \partial I \times B^2$ to $\partial D_m^2 \times\partial E_m^{c\ast}\times  D_n^2 \subset \partial D_m^2 \times E_m^{c\ast}\times  D_n^2$.
The map $w^{-1}\cup v$ induces a homeomorphism from $(D^2 \times \partial I \times B^2) \cup_{\lambda} (D_m^2 \times E_m^{c\ast}\times \partial D_n^2)$ to
\begin{equation}\label{pieceB}
(D^2 \times \partial I \times B^2) \cup_{\hat{{\rm id}}} (\partial D^2_m \times E^{c\ast}_m \times D^2_n).
\end{equation}

\begin{figure}[htbp]
\centering
\includegraphics[scale=0.68]{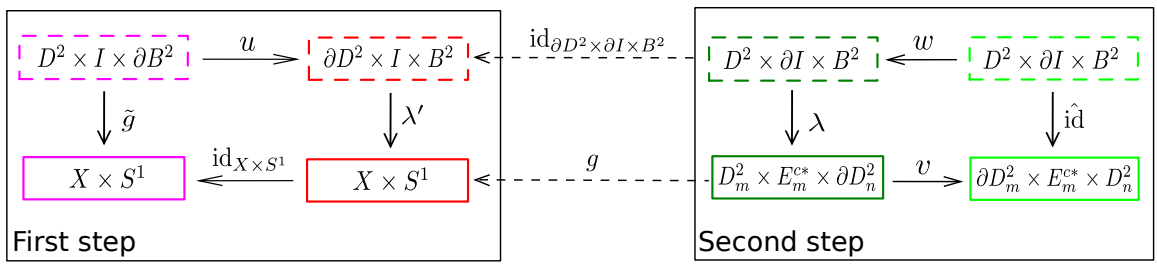}
\caption{Replacing pieces of $\Sigma(K^{m,n})$}
\label{maps}
\end{figure}

Finally we rearrange the decomposition \eqref{sigma} of $\Sigma(\Kmn)$ to the decomposition \eqref{sigma2}  using \eqref{pieceA} and \eqref{pieceB},
 where $\mu$ and $\mu^{\prime}$ are chosen as  $\mu = \hat{u}^{-1} \circ {\rm id}_{\partial D^2 \times \partial I \times  B^2} \circ \hat{w}$ and $\mu^{\prime} = {\rm id}_ {\partial X \times S^1} \circ g \circ \hat{v}^{-1}$.
 Note that these maps are the composites of the maps in the first and second rows in Figure~\ref{maps}.  
 We can check that $\mu$ corresponds to the matrix
$$
\begin{pmatrix}
0 & 1 \\
-1 & 1 \\
\end{pmatrix}^{-1}
 \begin{pmatrix}
1 & 0 \\
0 & 1 \\
\end{pmatrix}
 \begin{pmatrix}
0 & 1 \\
-1 & 0 \\
\end{pmatrix}
=
\begin{pmatrix}
1 & 1 \\
0 & 1 \\
\end{pmatrix}$$
and  
 $\mu^{\prime}$ corresponds to the matrix
  $$
  \begin{pmatrix}
1 & 0 \\
0 & 1 \\
\end{pmatrix}
  \begin{pmatrix}
\ve m & -\ve n \\
 \beta & \alpha  
\end{pmatrix}
\begin{pmatrix}
0 & -1 \\
1 & 1 
\end{pmatrix}^{-1}
 =
  \begin{pmatrix}
\ve (m+n) & \ve m \\
-\alpha +  \beta & \beta 
\end{pmatrix}.$$
By changing the orientations of $\theta_1$ and $\theta_2$ if necessary as before, the latter matrix is replaced by 
$\begin{pmatrix}
\ve^{\prime}(m + n)& \ve^{\prime} m \\
\alpha - \beta  & -\beta 
\end{pmatrix},
$
 which can be written as
$\begin{pmatrix}
\ve^{\prime}(m + n)& -\ve^{\prime} n \\
\alpha - \beta & -\alpha 
\end{pmatrix}
\begin{pmatrix}
1 & 1 \\
0 & 1 
\end{pmatrix}$.
Therefore the map $\mu \cup \mu^{\prime}$ is nothing but the inverse of the Gluck twist $\sigma$ on $S^4$ along the 2-knot $K^{m+n,n}$ that is the core of $(D^2 \times \partial I \times B^2) \cup_{\hat{\rm id}} (\partial D^2_m \times E^{c\ast}_m \times D^2_n)$ before the Gluck twist.

Since the Gluck twist $\sigma$ is isotopic to $\sigma^{-1}$, the 4-manifold on the right in Figure $\ref{arrange}$ is $\Sigma(K^{m+n,n})$.
On the other hand, since the gluing maps in Figure~\ref{maps} are commutative, 
the 4-manifold on the left in Figure~\ref{arrange} is diffeomorphic to that on the right. 
 Thus $\Sigma(\Kmn)$ is diffeomorphic to $\Sigma(K^{m+n,n})$.
All arguments in this proof can be applied to $K^{\ve n,\ve m}$ instead of $K^{m,n}$.
Finally, by using Euclidean algorithm as Pao did, we conclude
$$
\Sigma({K^{m,n}}) = \Sigma({K^{k,1}})\ \ \ \text{for some}\ \ k \in \mathbb{Z}. 
$$
Since $\Sigma({K^{k,1}})$ is diffeomorphic to $S^4$~\cite{G}, the assertion holds.
\end{proof}

\begin{thm}\label{thm2}
Let $(m,n) \in (\Z \setminus \{0\}) \times \N$ be a coprime pair. Then $K^{m+n, n}$ is obtained from $K^{m,n}$ by the Gluck twist along $K^{\ve n, \ve m}$, where $\ve= 1$ if $m>0$ and $\ve= -1$ if $m<0$.
\end{thm}
\begin{proof}
We decompose $S^4$ along $K^{\ve n,\ve m}$ into five pieces and glue them so that it realizes the Gluck twist of $S^4$ along $K^{\ve n,\ve m}$.
The glued 4-manifold $\Sigma(K^{\ve n,\ve m})$ is given as 
\begin{equation}\label{sigma3}
 \left((D^2 \times \partial I \times B^2) \cup (D^2 \times I \times \partial B^2)\right)
 \cup_{\tilde{\lambda} \cup {\tilde{\lambda}}^{\prime}} 
 \left((\partial D_m^2 \times E_n^{c\ast}\times D_n^2) \cup_e (X \times S^1) \right),
 \end{equation}
 where the gluing map $\tilde{\lambda} \cup {\tilde{\lambda}}^{\prime}$ is the one given in the decomposition \eqref{sigma} with
 exchanging the orders $m$ and $n$.
 Note that $\Sigma(K^{\ve n,\ve m})$ is $S^4$ by Theorem~\ref{thm1}.
The union $(D^2 \times \partial I \times B^2)\cup_{\lambda} (\partial D_m^2 \times E_n^{c\ast} \times D_n^2)$
 constitutes a neighborhood of a 2-knot in $\Sigma(K^{\ve n, \ve m})$ and our assertion is that the core of this union is $K^{m+n,n}$.
Since the decomposition~\eqref{sigma3} is given according to the orbit data $\{(S^3,K),m,n\}$, 
the complement of $(D^2 \times \partial I \times B^2)\cup (D^2 \times I \times \partial B^2)$ 
has the $S^1$-action such that 
$\{0\} \times I \times \partial B^2$ consists of exceptional orbits of order $n$.
Thanks to the classification of Fintushel and Pao, it is enough to show that this $S^1$-action extends to $ (D^2 \times \partial I \times B^2) \cup (D^2 \times I \times \partial B^2)$ with exceptional orbits of order $n$ and $m+n$.

Recall that, in case of \eqref{sigma}, $\lambda: D^2 \times \partial I \times \partial B^2 \to D^2 \times \partial I \times \partial B^2$ is defined by
$$
\lambda((r_1,\theta_1),x,\theta_2) = ((r_1,\theta_1 -\theta_2), x, \theta_2).
$$
In the current setting, $m$ and $n$ are exchanged, and the ``rolls'' of $D^2$ and $B^2$ are exchanged.
Hence the gluing map $\lambda$ after the Gluck twist is changed into the map
$\tilde{\lambda} : \partial D^2 \times \partial I \times B^2 \to \partial D^2 \times \partial I \times B^2$ defined by
$$
\tilde{\lambda}(\theta_1,x,(r_2,\theta_2)) = (\theta_1, x, (r_2,-\theta_1 + \theta_2)),
$$
instead of $\lambda$.
The $S^1$-action on $\partial D^2_m \times  E^{c\ast}_n \times D^2_n$ before the Gluck twist  is given as
$$
\psi \cdot (\theta_1, x , (r_2,\theta_2)) \mapsto (\theta_1 -\ve n\psi , x ,(r_2,\theta_2 + \ve m\psi)).
$$
Hence,  after the Gluck twist, 
the $S^1$-action on $\partial D^2 \times I \times  B^2 $ is given by
$$
\psi \cdot (\theta_1 , x ,(r_2, \theta_2)) \mapsto (\theta_1 -\ve^{\prime} n\psi , x , (r_2,\theta_2 + \ve^{\prime}(m + n) \psi)),
$$
where $\ve^{\prime}= 1$ if $m+n>0$ and $\ve^{\prime}= -1$ if $m+n<0$.
Thus the $S^1$-action on the complement of 
$(D^2 \times \partial I \times B^2) \cup (D^2 \times I \times \partial B^2)$
 extends to $D^2 \times I \times \partial B^2$ with exceptional orbits of order $m+n$.
Since $\partial (D^2 \times \partial I \times B^2)$ is the union of $\partial D^2 \times \partial I \times B^2$ and $D^2 \times \partial I \times \partial B^2$ and they have the $S^1$-actions with exceptional orbits of order $n$ and $m+n$, respectively, these actions extend to $D^2\times\partial I\times B^2$ canonically. 
Hence the $S^1$-action extends to the whole $\Sigma(K^{\ve n,\ve m})$.
\end{proof}

Note that we can regard the Gluck twist of $S^4$ along $K^{\ve n,\ve m}$ as the replacement of $E^{\ast}_n$ by $E^{\ast}_{m+n}$ in the orbit space $S^3$, see Figure~\ref{changeoforbits}.

\begin{figure}[htbp]
\centering
\includegraphics[scale=0.6]{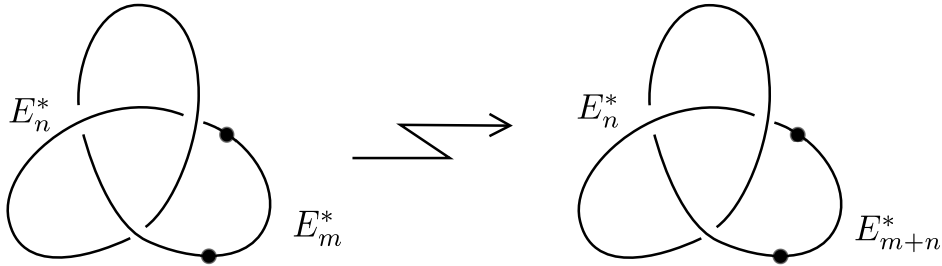}
\caption{The change of the orbit space by the Gluck twist along $K^{\ve n,\ve m}$}
\label{changeoforbits}
\end{figure}

We close this paper with a corollary.
In~\cite{G} Gordon showed that 
$K^{m,1}$ is not determined by its complement if $m$ is odd and the universal cover of $m$-fold cyclic branched cover of $S^3$ along $K$ is $\mathbb{R}^3$. 
This means that, due to Theorem~\ref{thm2}, $K^{m,1}$ is not equivalent to $K^{\ve^{\prime} m,\ve^{\prime} (m+1)}$ but the knot complements are homeomorphic.
More generally, Plotnick showed  any non-trivial fibered 2-knot $L$ with odd monodromy is not determined by its complement~\cite{Pl}.
In order to remove the assumption of Gordon's statement, 
he used an algebraic property of the knot group and ``special isometry'' on the second homotopy group of  a ``spun manifold" of the closure of the fiber.
Applying his observation to Theorem~\ref{thm2}, we may obtain the following corollary.
\begin{cor}\label{cor}
Assume that $m$ is odd and $K^{m,n}$ is non-trivial. Then $K^{m,n}$ and $K^{\ve^{\prime}m,\ve^{\prime}(m+n)}$ are not equivalent but their complements are homeomorphic, 
where $\ve^{\prime}= 1$ if $m+n>0$ and $\ve^{\prime}= -1$ if $m+n<0$.
\end{cor}
\begin{proof}
Consider $S^4$ with the $S^1$-action of orbit data $\{(S^3,K), m,n\}$ and set
$\Kmn = \En \cup F$ and $K^{\ve n,\ve m} = \Em \cup F$.
Applying Theorem~\ref{thm2} to the Gluck twist along $K^{m,n}$,
we see that the 2-knot $K^{\ve(m+n),\ve m}$ is obtained from $K^{\ve n, \ve m}$ by the Gluck twist along $K^{m,n}$
and the image of $K^{m,n}$ in $\Sigma(K^{m,n})$ is $K^{\ve^{\prime}m, \ve^{\prime}(m+n)}$.
Moreover, $K^{m,n}$ and $K^{\ve^{\prime}m,\ve^{\prime}(m+n)}$ have the same complement.
Since the order of the monodormy of $K^{m,n}$ is $m$,  
$K^{m,n}$ and $K^{\ve^{\prime}m,\ve^{\prime}(m+n)}$ are not equivalent by Theorem 6.2 in~\cite{Pl}  if $m$ is odd.
\end{proof}

\begin{rem}
A pair of two 2-knots $(L,L^{\prime})$ is a {\it Montesinos twin} if $L$ and $L^{\prime}$ meet transversely twice.
By definition, the twin $(K^{m,n}, K^{\ve n,\ve m})$ is a Montesinos twin.
A 2-knot $L$ is said to be {\it reflexive} if $L$ and the image of $L$ by the Gluck twist along $L$ are equivalent.
It is known by Hillman and Plotnick that $K^{m,n}$ is not reflexive if $K$ is a non-trivial torus or hyperbolic 1-knot, $m > n$ and $m \geq 3$~\cite{HP}.
By Theorem~\ref{thm2}, 
the twin $(K^{m,n}, K^{\ve n,\ve m})$ in $S^4$ is changed into the twin $(K^{\ve^{\prime}m,\ve^{\prime}(m+n)}, K^{\ve(m+n),\ve m})$ in $\Sigma(K^{m,n})$ by the Gluck twist along $K^{m,n}$ 
and the twin $(K^{\ve^{\prime}m,\ve^{\prime}(m+n)}, K^{\ve(m+n),\ve m})$ in $\Sigma(K^{m,n})$ is changed into
 $(K^{\ve^{\prime\prime}m,\ve^{\prime\prime}(2m+n)}, K^{\ve(2m+n), \ve m})$ in 
 $\Sigma(K^{\ve^{\prime}m,\ve^{\prime}(m+n)})$ by the Gluck twist along $K^{\ve^{\prime}m,\ve^{\prime}(m+n)},$
  where $\ve^{\prime\prime} = 1$ if $2m+n\geq 0$ and  $\ve^{\prime\prime} = -1$ if $2m+n < 0$.
 The composite of these Gluck twists is nothing but the product of two Gluck twists along $K^{m,n}$.
 Hence $K^{m,n}$ and $K^{\ve^{\prime\prime}m,\ve^{\prime\prime}(2m+n)}$ are equivalent. 
 This implies that the assumption $m > n$ in the above assertion in~\cite{HP} is not necessary.
\end{rem}
\begin{rem}
In~\cite{F}, we found a sufficient condition to distinguish branched twist spins by using the first elementary ideal when $m$ is even. 
However we cannot apply this argument in the case where the original 1-knots of branched twist spins are the same.
The above corollary gives a sufficient condition to distinguish such pairs of branched twist spins.
\end{rem}


\begin{thebibliography}{99}
\bibitem{Fi} R. Fintushel, \textit{Locally smooth circle actions on homotopy 4-spheres}, Duke Math. J. {\bf43} 
(1976), 63--70.

\bibitem{F} M. Fukuda, \textit{Branched twist spins and knot determinants}, Osaka. J. Math. {\bf 54} (2017), no. 4, 679--688.

\bibitem{Gl} H. Gluck, \textit{Embedding of two spheres in the four sphere}, Trans. Amer. Math. Soc. {\bf 104} (1962), 308--333.

\bibitem{G} M. C. Gordon, \textit{Knots in the 4-sphere}, Comment. Math. Helv. {\bf 51} (1976), 585--596.

%
%
%
%
%
%
\bibitem{HP} J. A. Hillman and S. P. Plotnick, \textit{Geometrically fibred two-knots}, Math. Ann. {\bf287} (1990), 259--273.

 
\bibitem{Ki}R. Kirby,(Ed.) \textit{Problems in Low-Dimensional Topology}, 1995.\\
 http://www.math.berkeley.edu/\verb|~|kirby/

\bibitem{MY} D. Montgomery and C. T. Yang, \textit{Groups on} $S^n$ \textit{with principal orbits of dimension n-3}, {\bf I}, {\bf II}, Illinois J. Math. {\bf4} (1960), 507--517.;{\bf5} (1961), 206--211.
%
\bibitem{OR} P. Orlik and F. Raymond, \textit{Actions of} $SO(2)$ \textit{on 3-manifolds}, 
Proc. Conf. on Transformation Groups (New Orleans, 1967), Springer-Verlag, New York (1968), 297--318.

\bibitem{Pa} P. S. Pao, \textit{Non-linear circle actions on the 4-sphere and twisting spun knots}, Topology {\bf 17} (1978), 291--296.
%
\bibitem{Pl} S. Plotnick, \textit{Equivaliant intersection forms, knots in} $S^4$, \textit{and rotations in 2-spheres}, Trans. Amer. Math. Soc. {\bf296} (1986), 543--575.

\bibitem{R} F. Raymond, \textit{Classification of the actions of the circle on 3-manifolds}, Trans. Amer. Math. Soc. {\bf131} (1968), 51--78.

\bibitem{S} H. Seifert, \textit{Topologie Dreidimensionaler Gefaserter R\"aume}, Acta Math. {\bf 1} (1933), 147--238. 

\bibitem{Z} E. C. Zeeman, \textit{Twisting spun knots}, Trans. Amer. Math. Soc. {\bf 115} (1965), 471--495.
%
%
%
%
\end{thebibliography}
\end{document}